 \documentclass[a4paper,11pt]{article}
 \pdfoutput=1
 \usepackage{amsmath, amssymb, latexsym, amscd, amsthm,amsfonts,amstext}
 \usepackage[mathscr]{eucal}
 \usepackage{graphicx}
 \usepackage{subfig}

 \newtheorem{theorem}{Theorem}[section]
 \newtheorem{lemma}[theorem]{Lemma}
 \newtheorem{definition}[theorem]{Definition}
 \newtheorem{algorithm}[theorem]{Algorithm}
\newtheorem{remark}{Remark}[section]

    \newtheoremstyle{example}{\topsep}{\topsep}%
      {}
      {}
      {\bfseries}
      {.}
      {5pt}
      {\thmname{#1}\thmnumber{ #2}\thmnote{ #3}}

    \theoremstyle{example}

 \title{Convergence rates of recursive Newton-type methods for multifrequency scattering problems}
 \author{Mourad Sini$^a$ and Nguyen Trung Th\`anh$^b$\\
$^a$Johann Radon Institute for Computational and Applied Mathematics (RICAM),\\
 Austrian Academy of Sciences, Linz, Austria. Email: mourad.sini@oeaw.ac.at \\
$^b$Department of Mathematics \& Statistics, University of North Carolina at Charlotte,\\
Charlotte 28223, NC, USA. Email: tnguy152@uncc.edu.  }
\date{}

\begin{document}

 \maketitle

\begin{abstract}
 We are concerned with the reconstruction of a sound-soft obstacle using far field measurements of the scattered waves associated 
 with incident plane waves sent from one direction but at multiple frequencies. We define, for each frequency, the observable shape as
 the one which is described by finitely many modes and produces a far field pattern close to the measured one. In the first step, we propose a recursive Newton-type method 
 for the reconstruction of the observable shape at the highest frequency knowing an estimate of the observable shape at the lowest frequency. 
 We analyze its convergence and derive its convergence rate in terms of the frequency step, the number of the Newton iterations and the noise level.
 In the second step, we design a multilevel Newton method which has the same convergence rate as the one described in the first step but avoids the need of a 
 good estimate of the observable shape at the lowest frequency and a small frequency step (or a large number of Newton iterations).
 The performances of the proposed algorithms are illustrated with numerical results using simulated data. 
\end{abstract}


\textbf{AMS classification codes:} 35R30, 65N21, 78A46.

%
%
%
%

\section{Introduction}\label{sec:ps}


We consider the problem of reconstructing the shape of a two-dimensional sound-soft acoustic obstacle using far field measurements associated with incident plane waves 
sent from only one incident direction but at multiple frequencies. The forward scattering problem can be represented by the following two-dimensional Dirichlet boundary value problem
\begin{eqnarray}
&&\Delta u(x) + k^2 u(x) = 0, \ x\in   \mathbb R^2\setminus \bar
D,\label{eq:1}\\
&&u(x) = 0, \ x\in \partial D,\label{eq:2}\\
&&\lim\limits_{|x| \to \infty } \sqrt{|x|} \left[ \frac{\partial
u^s(x)}{\partial |x|} - ik u^s(x)\right] = 0,\label{eq:3}
\end{eqnarray}
where $k$ is the wavenumber, $u$ is the total wave and $u^s := u - u^i$ is the
scattered wave.  Here, $u^i$ is the incident plane wave given by $u^i(x) :=
e^{ik
x\cdot \theta}$ with $\theta\in   \mathbb S^1 := \{ x\in   \mathbb R^2: |x|= 1\}$ being the direction of
incidence. The well-posedness of the problem \eqref{eq:1}--\eqref{eq:3} for each wavenumber $k$ is
well-known under the assumption that $\partial D$ is Lipschitz (see, e.g., \cite{Mclean:2000}). Moreover, we have the following asymptotic
behavior of
the scattered field $u^s$ at infinity
\begin{equation}\label{eq:4}
u^s(x) =  \frac{e^{ik |x|}}{\sqrt{|x|}}u^\infty(\hat x) +
O(|x|^{-3/2}), \ |x| \to \infty,
\end{equation}
where $\hat x := x/|x|$ and $u^\infty$ is an analytic function on $  \mathbb S^1$ referred
to as the \textit{far field pattern} of the scattered field $u^s$.

The inverse problem we investigate here is \textit{to reconstruct the
obstacle $D$
from measured far field patterns $u^\infty(\hat x,k),\ \hat x\in
  \mathbb S^1$, for \textbf{one} direction of incidence $\theta\in
\mathbb S^1$ and multiple wavenumbers $k$ in the interval
$[k_{l},k_{h}]$ ($0< k_{l} < k_{h}$)}. Here we denote the far field pattern by $u^\infty(\hat x, k)$ to emphasize its dependence on the
wavenumber $k$.

Let us recall some known results concerning this problem. It has a unique solution if a band of wavenumbers $[k_{l},k_{h}]$ is used, see, e.g., \cite{Ramm:1992}. 
If the measurements correspond to a finite number of frequencies, as we consider in this paper, then the uniqueness of the solution is guaranteed if the lowest frequency is small enough, see, e.g., \cite{C-S:IMA1983,Gintides:IP2005}. 
For local uniqueness at each frequency, we refer to  \cite{S-U:PAMS2004}. If more \textit{a priori}  information about the obstacle's shape is available, then some global uniqueness results at an arbitrary but 
fixed frequency have been published. For example, if the obstacles are polygonal, see \cite{A-R:PAMS2005,C-Y:CAMSB2004} and if thes obstacles are nowhere analytic, see \cite{H-N-S:pre2010}. 
Regarding the stability issue, \textit{loglog} stability estimates are given in \cite{Isakov:2006} and an improved \textit{log} stability estimate is shown in \cite{S-S:IPI2008}. 
In the high frequency regime, a conditional asymptotic H\"older stability estimate in the part of the boundary $\partial D$, of a convex obstacle $D$, illuminated by the 
incident plane wave $u^i$ is obtained in \cite{S-T:IPI2012}.

The main advantage of using multifrequency data is that it can help to obtain accurate reconstructions without the need for a good initial guess. 
Let us explain the reasons why we can expect these two features. The first one is that since the size of the domain in which the inverse problem is uniquely solvable is inversely proportional to the used frequency, 
the one in which the objective functional has a unique minimum is also inversely proportional to that frequency, as proved in \cite{S-T:IPI2012}. Due to this fact, if the lowest frequency used is small enough,  
it is not necessary to start from a good initial guess in solving for its unique minimum. 
The second one is related to the fact that, as discussed in \cite{B-I:RS1997}, for each frequency the dimension of the retrievable information is small which is due to the instability of the original problem. 
Therefore, at each frequency, we only need to choose a relatively small number of unknown parameters in the shape representation, 
which reduces the instability issue of the reconstruction problem. The third one is that at high frequencies the problem becomes more stable, i.e.~more details of the obstacle can be reconstructed. However,  
there are more local mimima of the objective functional. Using the reconstruction result at a lower frequency helps to avoid getting a false local minimum. 

Different reconstruction methods using multifrequency data have been proposed in the last two decades or so. The first type of method is known as frequency-hopping algorithms 
which use the reconstruction at a frequency as an initial guess at a higher frequency with the hope that this initial guess falls within the convergence domain of the objective 
functional. Several numerical results, using either simulated data, see e.g.~\cite{B-T:JCM2010, C-L:IEEE1995, Chen:IP1997, S-T:IPI2012}, or experimental data, 
see, e.g.~\cite{B-B-P-S-S:JEWA2000, T-B-L-H:IP2001, T-B-L-H:IEEETGRS2001}, have been demonstrated. However, convergence of this type of algorithms was only investigated 
in \cite{B-T:JCM2010, S-T:IPI2012} for the so-called recursive linearization algorithm (RLA) proposed in \cite{Chen:IP1997}. 
Another type of methods using multifrequency/multiwaves data, related to the sampling methods, can be found in \cite{Griesmaier:IP2011, G-C-B:IP2010, Potthast:IP2010}.

Inspired by the presentation in \cite{Chen:IP1997}, we define, for each frequency, the observable shape as
 the one which is described by finitely many modes and produces a far field pattern close to the measured one. Our goal then is to reconstruct the observable shape at the highest 
 available frequency $\kappa_h$. The link between this observable shape and the true one is related to the stability issue, see \cite{S-T:IPI2012} and section \ref{discussion--} for more explanation.
To achieve this goal, we proceed as follows. 

\begin{itemize}
\item First, we propose a projected recursive Newton method for solving this inverse problem. 
The idea is to use a certain number of Newton iterations at each frequency, starting from the lowest one, and then the reconstruction is used to linearize the problem 
at the next higher frequency. We prove the convergence rate of this algorithm, see section \ref{sec:RLA}, which shows a significant improvement compared to the linear 
convergence rate of the RLA obtained in \cite{B-T:JCM2010,S-T:IPI2012}. We investigate both noiseless and noisy data.

\item Second, a multi-level Newton method is proposed and its convergence rate is also investigated. The main idea of this method is to divide the whole frequency 
set into subsets and each of them are treated using the recursive Newton algorithm of section \ref{sec:RLA}. The difference between these two methods is that in multi-level 
Newton method the regularization parameter associated with different frequency subsets can be chosen to be different whereas in the original recursive Newton method this parameter 
is fixed at all frequencies. This adaptive choice of the regularization parameter allows us to obtain the same convergence rate as the previous algorithm but with less restrictive 
requirement on the accuracy of the reconstruction at the lowest frequency.  This topic is discussed in section \ref{sec:mNm}.  Related to this approach, we cite the work \cite{dHQS:2012} which also 
investigates a multi-level projected steepest descent method in Banach spaces with a discrepancy principle used for stopping the iterative process at each 
frequency. 
\end{itemize}

Finally, we show in section \ref{sec:num} some numerical results using simulated data to demonstrate the performance of the aforementioned algorithms. 
Our numerical results are consistent with the theoretical analysis of sections \ref{sec:RLA} and \ref{sec:mNm}.

Concerning the choice of the first guess at the starting frequency as well as the reconstruction accuracy, we refer the reader to section 3 of \cite{S-T:IPI2012}.

\section{A projected recursive Newton method and its convergence rate}\label{sec:RLA}

In this work, we consider the case of star-shaped obstacles whose boundary $\partial D$ can be represented by
\begin{equation}\label{eq:ps1}
  \partial D = \{ x(t)\in   \mathbb R^2: x(t) =  x^0 + r(t) (\cos t,\sin t) , t\in [0, 2\pi]\},
\end{equation}
where $x^0$ is a given internal point of $D$ in $  \mathbb R^2$ and
the radial function $r$ is positive in $[0,2\pi]$ with $r(0) = r(2\pi)$. In the following, we denote by $D(r)$ to indicate the dependence of
the obstacle on its radial function $r$. For each wavenumber $k$, we define the \textit{boundary-to-far field operator} (or \textit{far field operator}, in short) $F(\cdot,k)$ which maps each radial function $r$ to the far field pattern $u^\infty(\cdot,k,r)$ of the forward scattering problem \eqref{eq:1}--\eqref{eq:3} with $D = D(r)$. In this paper, we assume that the shape is of class $C^3$,
i.e., the $2\pi$-periodic extension of the radial function $r$ from $[0,2\pi]$ to $\mathbb R$ belongs to $C^3(\mathbb R)$. This smoothness guarantees the regularity of the derivatives of the far field operator used in section \ref{sec:conv}. We denote by $X_{ad}$ the set of radial functions of this $C^3$-class starlike shapes. This set is considered as the admissible set in our algorithm.

Let $X$ be a Hilbert space which contains the admissible set $X_{ad}$. In this work, we choose this space to be $X:= L_2[0,2\pi]$. However other spaces can be used as well. Therefore, for generality, in the following we use the notation $X$ instead of $L_2$. 
The derivative of $F$ with respect to the radial function, $\partial_r F(r,k)$, is defined by
\begin{equation*}
  \partial_r F(r,k)a :=
  \lim\limits_{\epsilon\to 0} \frac{F(r + \epsilon a,k) - F(r,k)}{\epsilon},
\end{equation*}
for $a\in X$. Note that $\partial_r F(r,k)$ is an injective linear operator from $X$ to $L_2(\mathbb S^1)$ for $r\in X_{ad}$, see \cite{Kirsch:IP1993,C-K:1998}. We also refer to these references for its characterization.

In the following sections, we denote by $u^{\infty,\delta}_m(\cdot,k) \in L_2(\mathbb S^1)$ the noisy measured far field pattern at the wavenumber $k$ with additive random noise of magnitude (noise level) $\delta \ge 0$. We define the operator $\tilde F_\delta$ from $X_{ad}$ to $L_2(\mathbb S^1)$ by $\tilde F_\delta(r,k) := F(r,k) - u^{\infty,\delta}_m(\cdot,k)$. The norms in $X$ and $L_2(\mathbb S^1)$ are denoted by $\|\cdot \|_X$ and $\|\cdot \|_2$, respectively.

\subsection{Description of the algorithm}

Suppose that the far field pattern is measured at the discrete set of frequencies $k_n := k_{l} + n\Delta k, n = 0, 1,\dots, N,$ with $\Delta k = \frac{k_{h} - k_{l}}{N}$. 
Consider a set of increasing subspaces $X_0 \subset X_1 \subset\cdots\subset X_N$ of $X$. The choice of these subspaces will be discussed later in section \ref{sec:conv}. We denote by $P_n$ the orthogonal projection of
$X$ onto $X_n$, $n = 0, 1, \dots, N$. Assume further that we have a rough approximation $r_0 \in X_0$ of the exact radial function at the lowest frequency $k_0 = k_l$, 
which can be obtained by minimizing the least-squares objective functional as described in the first step of \cite{S-T:IPI2012}. Given an integer $J > 0$ and an approximation $r_n\in X_n$ of
the radial function at wavenumber $k_n$, we denote by $r^0_{n+1} := r_n$ and consider $J$ Newton iterations at wavenumber $k_{n+1}$ as follows: 
$r^{j+1}_{n+1} := r^{j}_{n+1} + P_{n+1}\Delta r^j_{n+1}$, $j = 0, \dots, J-1,$ with $\Delta r^j_{n+1}$ being the solution of the regularized least-squares 
minimization problem
\begin{equation}\label{eq:rrla}
  \Delta r^j_{n+1} := \text{argmin}_{\Delta r} \left\{\frac{1}{2} \|\tilde F_\delta(r^j_{n+1},k_{n+1}) + \partial_r F(r^j_{n+1},k_{n+1})\Delta r\|_2^2 +\frac{1}{2} \alpha\|\Delta r\|_X^2 \right\}
\end{equation}
with $\alpha > 0$. The solution to \eqref{eq:rrla} is given  by
\begin{equation*}
  \Delta r^j_{n+1} =  - [\alpha I + (A^j_{n+1})^*A^j_{n+1}]^{-1}(A^j_{n+1})^* \tilde F_\delta(r^j_{n+1},k_{n+1}),
\end{equation*}
where $A^j_{n+1} := \partial_r F(r^j_{n+1},k_{n+1})$ and $(A^j_{n+1})^*$ is its adjoint operator. Hence,
\begin{equation}\label{eq:rla1-2}
  r^{j+1}_{n+1} := r^{j}_{n+1} - P_{n+1}[\alpha I + (A^j_{n+1})^*A^j_{n+1}]^{-1}(A^j_{n+1})^* \tilde F_\delta(r^j_{n+1},k_{n+1}), j = 0, \dots, J-1,
\end{equation}
Since $r^0_{n+1} = r_n\in X_n\subset X_{n+1}$, the approximations $r^j_{n+1}$ also belong to the subspace $X_{n+1}$ for $j = 1, \dots, J$. We choose $r_{n+1} := r^J_{n+1} \in X_{n+1}$ 
as the reconstruction at the wavenumber $k_{n+1}$. This process 
is repeated until the highest wavenumber $k_{N} = k_{h}$. The algorithm is summarized as follows.

\begin{algorithm}\label{alg:1}
$ $
\begin{itemize}
\item Given measured data  $u^{\infty,\delta}_m(\cdot,k)$ for $k =
k_0, \dots, k_N$, the parameter $\alpha > 0$ and the subspaces $X_n$, $n = 0,\dots, N$. 
\item Step 1: find an approximation $r_0 \in X_0$ at frequency $k_0$.
\item Step 2 (recurrence) For $n = 0,\dots, N-1$
    \begin{itemize}
    \item Set $r^0_{n+1} := r_n.$
    \item For $j = 0, \dots, J-1$
	\begin{itemize}
	\item Compute $\tilde F_\delta(r^j_{n+1},k_{n+1})$, $A^j_{n+1}$ and $(A^j_{n+1})^*$.
	\item Compute $r^{j+1}_{n+1} := r^{j}_{n+1} - P_{n+1}[\alpha I + (A^j_{n+1})^*A^j_{n+1}]^{-1}(A^j_{n+1})^* \tilde F_\delta(r^j_{n+1},k_{n+1})$.		
	\end{itemize}
    End (for $j$)
    \item Set $r_{n+1} := r^J_{n+1}$.   
    \end{itemize}
End (for $n$).
\end{itemize}
\end{algorithm}

\begin{remark}~~~~~~~~~~~~~~~~~~~~~~~~~~~~~~~~~~~~~~~~~~~~~~~~~~~~~~~~~~~~~~~~~~~~~~~~~~~~~~~~~~~~~~~~~~~~~~~~~~~~~~~~~~~~~~~~~~~~

1. In the recursive linearization algorithm, as discussed in \cite{B-T:JCM2010, Chen:IP1997, S-T:IPI2012}, only one Newton step is used at each frequency. 
 In other words, the reconstruction at $k_{n+1}$ is chosen by $r_{n+1} := r^1_{n+1}$. 

2. The stopping criteria in Algorithm  \ref{alg:1} is related to a trade-off between the frequency step $\Delta k$ (or the number of $N$
of the used frequencies) and the number  $J$ of the Newton iterations to achieve a final error of the order $O(\delta^{\frac{2}{3}})$, see Remark \ref{trade-off}.
 \end{remark}

To implement Algorithm \ref{alg:1}, it is necessary to represent the radial function $r$ as a function of a finite number of parameters. Since any radial function $r(t)$ satisfies $r(0) = r(2\pi)$, it can be considered as a periodic function with the period of $2\pi$. Hence, we can represent it as the following Fourier
series
\begin{equation}\label{eq:con1}
  r(t) = \beta_0 + \sum\limits_{m = 1}^\infty (\beta_m\cos mt + \gamma_m\sin mt).
\end{equation}
We note that the Fourier coefficients $\beta_m$ and $\gamma_m$ converge to zero as $m\to \infty$. Their convergence rate depends on the smoothness of the function $r(t)$, see \cite{Folland:1992}. For each number $M\in   \mathbb N$, we define the cut-off approximation $r^M(t)$ of $r(t)$ by
\begin{equation}\label{eq:con2}
  r^M(t) := \beta_0 + \sum\limits_{m = 1}^M (\beta_m\cos mt + \gamma_m\sin mt).
\end{equation}
It is clear that, for large $M$, $r^M$ is different from $r$ just in high frequency modes which represent small details of the obstacle shape. Let us recall the notion of \textit{finite dimensional observable shapes} which was defined in \cite{S-T:IPI2012}. For a given value $\tilde \delta > \delta$, there exists a number $M_0(k)\in  \mathbb  N$ depending on $k$ such that $\|F(r^M,k) - F(r,k)\|_2 \le \tilde \delta -\delta$ for all $M \ge M_0(k)$. Consequently,  $\|F(r^M,k) - u^{\infty,\delta}_m(\cdot,k)\|_2 \le \tilde \delta$ for $M \ge M_0(k)$. Note that $M_0(k)$ also depends on $\tilde \delta$ and $\delta$, but we ignore these parameters since they are fixed throughout the paper.  From this analysis, we can simplify the inverse problem by determining the cut-off approximation $r^M$ (or its Fourier coefficients) instead of the radial function $r$ itself. By this simplification, the inverse problem becomes finite dimensional. 

In the following analysis, we choose the subspaces $X_n$ of $X$ containing all functions 
of the form \eqref{eq:con2} with $M := M_n$ depending on $n$, that means $X_n$ is spanned by the basis $\{1, \cos t, \sin t, \dots, \cos M_nt, \sin M_nt \}$. We denote by $X_n^+ := \{\varphi\in X_n: \varphi(t) > 0, \forall t\in [0,2\pi]\}$. Since functions in $X_n$ are smooth, we have $X_n^+ \subset X_{ad}$.

\begin{definition}\label{def:fdos}
  For each wavenumber $k$ and a given $\tilde \delta > \delta$, a \textit{finite dimensional observable shape} (or, in short, observable shape) $D(\tilde r(k))$  is defined as a domain of which the radial function $\tilde r(k)\in X_M^+$ for some $M\in   \mathbb N$  and the corresponding far field pattern $F(\tilde r(k),k)$ satisfies the condition $\|F(\tilde r(k),k) - u^{\infty,\delta}_m(\cdot,k)\|_2 \le \tilde \delta $.
\end{definition}

 By this definition, a finite dimensional observable shape basically produces the same measured data as the true one (up to the noise level) but usually has a simpler Fourier series. It is obvious that $D(r^M)$ is a finite dimensional observable shape of the obstacle $D(r)$ for $M\ge M_0(k)$. However, we should emphasize that, there may be several finite dimensional observable shapes which are very different from $D(r^M)$ due to the ill-posedness of the considered inverse problem. The question on how these observable shapes approximate the original one relates closely to the stability of the inverse problem which was discussed in \cite{S-T:IPI2012}.

As remarked in \cite{S-T:IPI2012}, we made use of the value $\tilde \delta$ instead of the noise level $\delta$ because if the latter is used, the finite dimensional observable shapes might not exist. However, it is possible to choose $\tilde \delta$  close to $\delta$ while $M_0(k)$ can still be chosen not too large. This can be explained using the Heisenberg's uncertainty principle in Physics on the resolution limit of
scattering problems. It says that, at a fix frequency, we cannot observe small details of the scatterer using noisy measurements of the far field pattern, regardless the noise magnitude. In other words, choosing too many Fourier modes does not help to improve the reconstruction accuracy but increases the instability of the reconstruction. Therefore, $M_0(k)$ should not be chosen too large. As shown in \cite{Chen:IP1997}, this resolution limit is about half of the wavelength for weak penetrable scatterers, see also \cite{A-G-K-L-S:2010,B-T:JCM2010}. Due to this  uncertainty principle, we also choose $r_n$, $n = 0,\dots, N$, in Algorithm \ref{alg:1} such that they contain finite numbers of Fourier modes.

\subsection{Convergence rate}\label{sec:conv}

Algorithm \ref{alg:1} requires an approximation $r_0$ of the true radial function at the lowest frequency. As proved in \cite{S-T:IPI2012}, we can only guarantee a ``good'' $r_0$ 
if the true obstacle $D$ is contained in the  disk $B(x^0,\frac{\pi}{k_{l}})$ centered at the point $x^0$ and  radius $\frac{\pi}{k_{l}}$. Therefore, we first assume that the 
unknown obstacle is within a given region and the lowest frequency $k_l$ is chosen so small that this region is contained inside the disk $B(x^0,\frac{\pi}{k_{l}})$. 
Moreover, in the sequel, we make the following assumptions about the radial functions of the observable shapes:

\textbf{Assumption 1}: The radial functions $\tilde r(k_n), n = 0,\dots, N,$ are bounded from below, i.e., there exists a constant $\tilde c> 0$ such that 
\begin{equation}\label{eq:ass1}
 \tilde c \le \|\tilde r(k_n)\|, \text{ for all } n,
\end{equation}
 where $\|\cdot\|$ represents the maximum norm. Since the observable shapes, roughly speaking, are approximations of the true one, 
 the assumption (\ref{eq:ass1}) requires that the size of the true obstacle is not too small. As indicated in Theorem \ref{the:con1}, this lower bound $\tilde c$ can be chosen comparable to the regularization parameter $\alpha$, see \eqref{eq:cond0}, which is reasonably small. That means, this assumption is not very restrictive. 

\textbf{Assumption 2}: There exists a constant $d_0 \ge 1$ such that 
\begin{equation}\label{eq:con5}
    \|\tilde r(k_{n+1}) -  \tilde r(k_{n})\|_X \le d_0 |k_{n+1} - k_n|, \forall n = 0,\dots, N-1.
\end{equation}
Roughly speaking, this assumption says that the observable shapes of two consecutive frequencies should not be too different. For more details about the validity of Assumption 2, see Remark 3 of \cite{S-T:IPI2012}.

For the following convergence analysis, we assume that the subspaces $X_{n}$, $n = 0,\dots, N$ are chosen such that they contain the radial functions $\tilde r(k_n)$ of the observable shapes. 
We denote by $\tilde A_n := \partial_r F(\tilde r(k_n),k_{n})$, $n = 0, 1, \dots, N$ and $\sigma_n$ the smallest singular value of $\tilde A_n$ restricted to $X_{n+1}$, $\tilde A_n\big\vert_{X_{n+1}}$, for $n = 0, \dots, N-1$. Since these operators are injective, we have $\sigma_n > 0$, $n = 0,\dots, N-1$. Finally we define
\begin{equation}\label{eq:sigma}
\sigma := \min \{\sigma_0,\dots, \sigma_{N-1} \}. 
\end{equation}

For the radial functions $\tilde r(k), k\in [k_l,k_h]$ associated with a given set of  observable shapes of $r$, we write the operator $\tilde F_\delta$ as
\begin{equation}\label{eq:F1}
\tilde F_\delta(r,k) = \tilde F(r,k) + f^{\tilde\delta}(\tilde r(k),k)
\end{equation}
 with $\tilde F(r,k) := F(r,k)  - F(\tilde r(k),k)$ and $f^{\tilde\delta}(\tilde r(k),k) := F(\tilde r(k),k) - u^{\infty,\delta}_m(\cdot,k)$. Note that $\|f^{\tilde\delta}(\tilde r(k),k)\|_2 \le \tilde\delta$. It is obvious that
\begin{equation}\label{eq:con3}
  \tilde F(\tilde r(k),k) = 0, \forall k\in [k_{l},k_{h}].
\end{equation}
Note that  $F(r,k)$ is twice continuously
differentiable (see Remark 1 of \cite{S-T:IPI2012}). Therefore, there exist some positive constants $d_i, i = 1,\dots, 4$, such that for all $r\in B(\frac{\pi}{k_l})$ and $k\in [k_l,k_h]$, we have
\begin{equation}\label{eq:con4}
  \begin{split}
    &\|\partial_r \tilde F(r,k)\|_{\mathcal L(X,Y)} \le d_1, \ \|\partial_k \tilde F(r,k)\|_2 \le d_2,\\
    &\|\partial^2_{rr} \tilde F(r,k)\|_{\mathcal L(X\times X,Y)} \le d_3,\ \|\partial^2_{kr} \tilde F(r,k)\|_{\mathcal L(X,Y)} \le d_4.
  \end{split}
\end{equation}

In this section, we need the following estimates concerning compact linear operators.
\begin{lemma}\label{le:spectral}
  Let $A$ be a compact linear operator from a Hilbert space $X$ to a Hilbert space $Y$ and $R_\alpha(A):= (\alpha I + A^*A)^{-1}A^*$ with $\alpha > 0$. Then
  \begin{eqnarray}
    &&\|(\alpha I + A^*A)^{-1}\|_{\mathcal L(X,X)} \le \frac{1}{\alpha},\label{eq:A1}\\
    &&\|R_\alpha(A)\|_{\mathcal L(Y,X)} \le \frac{1}{2\sqrt \alpha},\label{eq:A2}\\
    &&\|R_\alpha(A)A\|_{\mathcal L(X,X)} \le 1. \label{eq:A3}
  \end{eqnarray}
Moreover, if $\tilde A$ is also a compact linear operator from $X$ to $Y$, we have
\begin{equation}\label{eq:R}
  \|R_\alpha(A) - R_\alpha(\tilde A)\|_{\mathcal L(Y,X)} \le \frac{9}{4\alpha}\|A - \tilde A\|_{\mathcal L(X,Y)}.
\end{equation}
\end{lemma}

We first prove the following result for the case of noiseless data.

\begin{theorem}\label{the:con1}
 Assume that the radial functions $\tilde r(k_n), n = 0,\dots, N,$ of the observable shapes satisfy Assumptions 1 and 2. Let $X_n, n = 0,\dots, N,$ be the subspaces of $X$ containing these radial functions. Let $r_n, n = 0,\dots, N,$ be given by  Algorithm \ref{alg:1} with $\tilde F_\delta$ being replaced by $\tilde F$. Then for a fixed positive real number $\epsilon$, $0 < \epsilon < 3/(2+d_0)$, and for the regularization parameter $\alpha$ satisfying
\begin{equation}\label{eq:alpha}
   \alpha \le \dfrac{\epsilon \sigma^2}{3 - \epsilon},
\end{equation}
there exists an integer $N_0$ depending on $\epsilon$ and $\alpha$ such that  if
\begin{equation}\label{eq:cond0}
    \|\tilde r(k_{l}) - r_0\|_X \le d_0c_0\alpha < \tilde c,
\end{equation}
with 
\begin{equation}\label{eq:c0}
 c_0 :=  \frac{4\epsilon}{3d_3(9d_1+ \sqrt{\alpha})},
\end{equation}
then we have $r_{n}(t) > 0\  \forall  t\in [0,2\pi]$ and the following error estimate holds true
  \begin{equation}\label{eq:con380}
    \|\tilde r(k_{h}) - r_N\|_X \le C_1 \left(\frac{\epsilon(1+d_0)}{3}\right)^{J-1} \sqrt{\Delta k}, \forall N \ge N_0,
  \end{equation}
  where $C_1$ is a constant independent of $\alpha$ and $N$.
\end{theorem}

\begin{proof}
For $n = 0,\dots, N$ and $j = 0,\dots, J$, we denote by $e_n := \tilde r(k_n) - r_n$, $R^j_n := [\alpha I + (A^j_n)^*A^j_n]^{-1}(A^j_n)^*$ and $\tilde R_n := (\alpha I + \tilde A_n^*\tilde A_n)^{-1}\tilde A_n^*$. We also denote by  $e^j_{n+1} := \tilde r(k_{n+1}) - r^j_{n+1}$ for $j = 1, \dots, J; n = 1,\dots, N$.

 We first estimate $e^1_{n+1}$. Here we repeat some arguments of \cite{B-T:JCM2010,S-T:IPI2012}. It follows from \eqref{eq:rla1-2} that
\begin{equation}\label{eq:con8}
\begin{split}
  e^1_{n+1} =& \tilde r(k_{n+1}) - r^0_{n+1} + P_{n+1}R^0_{n+1} \tilde F(r^0_{n+1},k_{n+1})\\
      =& \tilde r(k_{n+1}) - \tilde r(k_{n}) +  \tilde r(k_{n}) - r^0_{n+1} - P_{n+1} \tilde R_n \tilde A_n e_n \\
      + &  P_{n+1} \left[\tilde R_n \tilde A_n e_n+ R^0_{n+1} \tilde F(r^0_{n+1},k_{n+1}) \right].
\end{split}
\end{equation}
We recall that $r^0_{n+1} = r_n$. Let us evaluate the right hand side. Firstly, the first two terms are bounded by \eqref{eq:con5}. Secondly, note that since $\tilde r(k_n), r_n\in X_n \subset X_{n+1}$, we have $ \tilde r(k_{n}) - r^0_{n+1} = P_{n+1} e_n$. The spectral theory implies that
\begin{equation}\label{eq:con16}
   \| P_{n+1}\left[ e_n - \tilde R_n \tilde A_n e_n \right]\|_X\le \frac{\alpha}{\alpha+ \sigma^2}\|e_n\|_X.
\end{equation}
Thirdly,  
\begin{equation}\label{eq:con10}
\begin{split}
  \tilde R_n \tilde A_n e_n+ R^0_{n+1} \tilde F(r^0_{n+1},k_{n+1})& = \tilde R_n[\tilde A_n e_n + \tilde F(r_n,k_n)] - (\tilde R_n - R^0_{n+1})\tilde F(r_n,k_n)\\
  &  + R^0_{n+1}[\tilde F(r_n,k_{n+1})- \tilde F(r_n,k_{n})].
\end{split}
\end{equation}
Using the Taylor expansion of $\tilde F(r_n,k_n)$ at $\tilde r(k_n)$ up to the second order, \eqref{eq:A2} and \eqref{eq:con3}--\eqref{eq:con4}, we have
\begin{equation}\label{eq:con11}
 \| \tilde R_n[\tilde A_n e_n + \tilde F(r_n,k_n)]\|_X \le \frac{d_3}{4\sqrt\alpha}\|e_n\|_X^2.
\end{equation}
On the other hand, it follows from Lemma~\ref{le:spectral} and \eqref{eq:con3}--\eqref{eq:con4} that
\begin{equation*}\label{eq:con12}
\begin{split}
 \| (\tilde R_n - R^0_{n+1})\tilde F(r_n,k_n)\|_X & \le \frac{9}{4\alpha}\|A^0_{n+1} - \tilde A_n\|_{\mathcal L(X,Y)} \|\tilde F(r_n,k_n) - \tilde F(\tilde r(k_n),k_n)\|_2\\
  & \le \frac{9d_1}{4\alpha}\|A^0_{n+1} - \tilde A_n\|_{\mathcal L(X,Y)} \|e_n\|_X.
\end{split}
\end{equation*}
From the definition of $A^0_{n+1}$ and $\tilde A_n$ we have
\begin{equation*}
\begin{split}
 \| A^0_{n+1} - \tilde A_n\|_{\mathcal L(X,Y)} \le & \|\partial_r\tilde F(r_n,k_{n+1}) - \partial_r\tilde F(r_n,k_{n})\|_{{\mathcal L(X,Y)}}\\
&  + \|\partial_r\tilde F(r_n,k_{n}) - \partial_r\tilde F(\tilde r(k_n),k_{n})\|_{{\mathcal L(X,Y)}}\\
  \le & \Delta k d_4 + d_3\|e_n\|_X.
\end{split}
\end{equation*}
Replacing this estimate into the above inequality we obtain
\begin{equation}\label{eq:con13}
  \|(\tilde R_n - R^0_{n+1})\tilde F(r_n,k_n)\|_X \le \frac{9d_1}{4\alpha}\left[\Delta k d_4 + d_3\|e_n\|_X \right] \|e_n\|_X.
\end{equation}
It follows from \eqref{eq:con4} that
\begin{equation}\label{eq:con14}
  \|R^0_{n+1}[\tilde F(r_n,k_{n+1})- \tilde F(r_n,k_{n})]\|_X \le \frac{d_2}{2\sqrt\alpha}\Delta k.
\end{equation}
Substituting \eqref{eq:con11}--\eqref{eq:con14} into \eqref{eq:con10}, we obtain
\begin{equation}\label{eq:con102}
\begin{split}
  \|P_{n+1}\left[\tilde R_n \tilde A_n e_n+ R^0_{n+1} \tilde F(r_n,k_{n+1}) \right] \|_X \le  \frac{d_2}{2\sqrt\alpha}\Delta k  + \left(\frac{9d_1d_3}{4\alpha} + \frac{d_3}{4\sqrt\alpha} \right)\|e_n\|_X^2.
\end{split}
\end{equation}
By combining \eqref{eq:con5}, \eqref{eq:con16} and \eqref{eq:con102} we have
\begin{equation}\label{eq:con32}
\begin{split}
  \|e^1_{n+1}\|_X & \le \Delta k\left( d_0 + \frac{d_2}{2\sqrt\alpha}\right) +  \frac{\alpha}{\alpha+ \sigma^2}\|e_n\|_X \\
  & + \frac{9d_1d_4}{4\alpha}\Delta k \|e_n\|_X + \left(\frac{9d_1d_3}{4\alpha} + \frac{d_3}{4\sqrt\alpha} \right)\|e_n\|_X^2.
\end{split}
\end{equation}
Let us estimate the right hand side of \eqref{eq:con32}. First, it follows from \eqref{eq:alpha} that
\begin{equation}\label{eq:con18}
  \frac{\alpha}{\alpha+ \sigma^2}\le \frac{\epsilon}{3},
\end{equation}
Next, if $\|e_n\|_X \le d_0c_0 \alpha$, from \eqref{eq:c0} we have
\begin{equation}\label{eq:con19}
  \left(\frac{9d_1d_3}{4\alpha} + \frac{d_3}{4\sqrt\alpha} \right)\|e_n\|_X = \frac{\epsilon}{3c_0 \alpha} \|e_n\|_X \le \frac{d_0\epsilon}{3}.
\end{equation}
For the chosen $\alpha$, we can also choose a number $N_0 = N_0(\alpha)$ such that for all $N\ge N_0$, we have
\begin{equation}\label{eq:con20}
\frac{9d_1d_4}{4\alpha}\Delta k \le \frac{\epsilon}{3}
\end{equation}
 and
\begin{equation}\label{eq:con212}
   \Delta k\left( 1 + \frac{d_2}{2d_0\sqrt\alpha}\right) \le \left[1 - \frac{\epsilon}{3}(2+d_0) \right] c_0\alpha.
\end{equation}
Note that the right hand side is positive. It follows from \eqref{eq:con32}--\eqref{eq:con212} that
$$ \|e^1_{n+1}\|_X \le\left[1 - \frac{\epsilon}{3}(2+d_0) \right]d_0c_0\alpha + \frac{\epsilon}{3}(2+d_0) \|e_n\|_X \le d_0c_0 \alpha .$$ 
Next, we estimate $e^{j+1}_{n+1}$ for $j = 1, \dots, J-1$. We rewrite them in the form
\begin{equation}\label{eq:con8-2}
\begin{split}
  e^{j+1}_{n+1} =&  \tilde r(k_{n+1}) - r^{j}_{n+1} + P_{n+1} R^{j}_{n+1} \tilde F(r^j_{n+1},k_{n+1})\\
      =& P_{n+1} \left[e^{j}_{n+1} - \tilde R^j_{n+1} \tilde A_{n+1} e^j_{n+1}\right] + P_{n+1}\tilde R^j_{n+1} \tilde A_{n+1} e^j_{n+1} + P_{n+1}R^j_{n+1} \tilde F(r^j_{n+1},k_{n+1}).
\end{split}
\end{equation}
By the same arguments as above, we obtain
\begin{equation*}
  \|e^{j+1}_{n+1}\|_X \le \frac{\alpha}{\alpha+ \sigma^2}\|e^j_{n+1}\|_X + \left(\frac{9d_1d_3}{4\alpha} + \frac{d_3}{4\sqrt\alpha} \right)\|e^j_{n+1}\|_X^2.
\end{equation*}
Hence, under the same conditions \eqref{eq:con18} and \eqref{eq:con19}
\begin{equation}\label{eq:con22-2}
  \|e^{j+1}_{n+1}\|_X \le  \frac{\epsilon}{3}(1+d_0) \|e^j_{n+1}\|_X < \|e^j_{n+1} \|_X, j = 1,\dots, J-1.
\end{equation}
Therefore, if $\|e_0\|_X \le d_0c_0\alpha$, we can prove by recurrence that $\|e^j_{n}\|_X\le d_0c_0\alpha$ for $j = 1,\dots, J$ and $n = 1, \dots, N$. From this it is clear that 
$$
r^j_n(t) \ge \tilde r(k_n,t) - e^j_n(t) \ge \tilde c - d_0c_0\alpha > 0 \text{ for all } j \text{ and } n.
$$
 Hence, $r_n(t) > 0$ for all $n$.   Moreover, 
\begin{equation*}\label{eq:con22}
  \|e_{n+1}\|_X \le \left(\frac{\epsilon(1+d_0)}{3}\right)^{J-1}\left[\Delta k\left( d_0 + \frac{d_2}{2\sqrt\alpha}\right) + \frac{\epsilon(2+d_0)}{3} \|e_n\|_X \right], \forall n = 0,\dots, N-1.
\end{equation*}
Consequently,
\begin{equation}\label{eq:con23}
\begin{split}
  \|e_{N}\|_X \le \left(\frac{\epsilon(1+d_0)}{3}\right)^{J-1} &\Delta k\left( d_0 + \frac{d_2}{2\sqrt\alpha}\right) \frac{1}{1 - \left(\frac{\epsilon(1+d_0)}{3}\right)^{J-1}\frac{\epsilon(2+d_0)}{3}}\\
 & + \left(\frac{\epsilon(1+d_0)}{3}\right)^{(J-1)N}\left(\frac{\epsilon(2+d_0)}{3}\right)^{N}\|e_0\|_X\\
   \le \left(\frac{\epsilon(1+d_0)}{3}\right)^{J-1} &\frac{\Delta k}{\sqrt\alpha}\left\{ \left( d_0 \frac{\sqrt{\epsilon}\sigma}{\sqrt{3-\epsilon}} + \frac{d_2}{2}\right)\frac{1 }{1 - \left(\frac{\epsilon(1+d_0)}{3}\right)^{J-1}\frac{\epsilon(2+d_0)}{3}} \right. \\
 & \left. +  \left(\frac{\epsilon(1+d_0)}{3}\right)^{(J-1)(N-1)}\frac{N \left(\frac{\epsilon(2+d_0)}{3}\right)^{N}}{k_{h}-k_{l}} d_0c_0 \left(\frac{\sqrt{\epsilon}\sigma}{\sqrt{3-\epsilon}}\right)^3\right\} .
\end{split}
\end{equation}
From \eqref{eq:c0} we can see that $c_0$ is bounded from above by 
\begin{equation}\label{eq:c02}
c_0 \le \dfrac{4\epsilon}{27 d_1 d_3}.   
\end{equation}
 Moreover, for a fixed frequency interval $[k_l,k_h]$, $N\epsilon^{N}$ is bounded in terms of $N$. Therefore, there exists a constant $C^* > 0$ independent of $N$ and $\alpha$ such that
\begin{equation}\label{eq:con24}
\begin{split}
& \left( d_0 \frac{\sqrt{\epsilon}\sigma}{\sqrt{3-\epsilon}} + \frac{d_2}{2}\right)\frac{1 }{1 - \left(\frac{\epsilon(1+d_0)}{3}\right)^{J-1}\frac{\epsilon(2+d_0)}{3}} \\
& +  \left(\frac{\epsilon(1+d_0)}{3}\right)^{(J-1)(N-1)}\frac{N \left(\frac{\epsilon(2+d_0)}{3}\right)^{N}}{k_{h}-k_{l}} d_0c_0 \left(\frac{\sqrt{\epsilon}\sigma}{\sqrt{3-\epsilon}}\right)^3 \le C^*.
\end{split}
\end{equation}
%
On the other hand, it follows from \eqref{eq:con20} that $\sqrt{\frac{\Delta k }{\alpha}} \le \sqrt{\frac{4\epsilon}{27d_1d_4}}$. 
Replacing these inequalities into \eqref{eq:con23} we obtain \eqref{eq:con380} with $C_1 = C^*\sqrt{\frac{4\epsilon}{27d_1d_4}}$.
\end{proof}

In the case of noisy data, we have the following result.

\begin{theorem}\label{th:con2}
  Assume that the radial functions $\tilde r(k_n)$ and the subspaces $X_n, n = 0, \dots, N,$  are as in Theorem \ref{the:con1}. Let $r_n, n = 0,\dots, N,$ be given by  Algorithm \ref{alg:1}. For fixed positive real numbers $\epsilon$, $\xi$, $0 < \epsilon < 3/(2+d_0), 0 < \xi < 1$, and for the parameters $\alpha$  and $c_0$ satisfying \eqref{eq:alpha} and \eqref{eq:c0} respectively, we define the positive parameter $\tilde\delta_0$ by
  \begin{equation}\label{eq:delta0}
   \tilde\delta_0 := 2\xi \left[1 - \frac{\epsilon(2+d_0)}{3}\right] d_0c_0 \alpha^{3/2}.
  \end{equation}  
Then there exists an integer $N_0$ independent of $\tilde \delta$ such that if \eqref{eq:cond0} is satisfied,  we have $r_{n}(t) > 0\  \forall  t\in [0,2\pi]$ and the following error estimate holds true
 \begin{equation}\label{eq:con38}
    \|\tilde r(k_{h}) - r_N\|_X \le C_2 \tilde \delta^{2/3} +  \left(\frac{\epsilon(1+d_0)}{3}\right)^{J-1} C_1 \sqrt{\Delta k} ,\ \forall N \ge N_0
  \end{equation}
  for every $\tilde\delta \le \tilde\delta_0$, where $C_1$ is as in Theorem \ref{the:con1} and $C_2$ is a constant independent of $\tilde\delta$, $\alpha$ and $N$.
\end{theorem}
\begin{proof}
 Using \eqref{eq:F1} we can rewrite the error as
\begin{equation}\label{eq:con01}
  e^{j+1}_{n+1} = \tilde r(k_{n+1}) - r^j_{n+1} + P_{n+1}R^j_{n+1} \tilde F(r^j_{n+1},k_{n+1}) + P_{n+1} R^j_{n+1} f^{\tilde\delta}(\tilde r(k_{n+1}),k_{n+1})
\end{equation}
for $j = 0,\dots, J-1$. It follows from Lemma \ref{le:spectral} that
\begin{equation}\label{eq:con143}
 \|P_{n+1}R^j_{n+1}
 f^{\tilde\delta}(\tilde r(k_{n+1}),k_{n+1})\|_X \le \frac{\tilde \delta}{2\sqrt\alpha}.
\end{equation}
Using the estimates \eqref{eq:con32} and \eqref{eq:con22-2} for the noiseless case, from \eqref{eq:con01}--\eqref{eq:con143} we have
\begin{equation}\label{eq:con322}
 \begin{split}
   \|e^1_{n+1}\|_X & \le \Delta k\left( d_0 + \frac{d_2}{2\sqrt\alpha}\right) +\frac{\tilde\delta}{2\sqrt\alpha}+  \frac{\alpha}{\alpha+ \sigma^2}\|e_n\|_X \\
   & + \frac{9d_1d_4}{4\alpha}\Delta k \|e_n\|_X + \left(\frac{9d_1d_3}{4\alpha} + \frac{d_3}{4\sqrt\alpha} \right)\|e_n\|_X^2.
 \end{split}
\end{equation}
And for $j = 0,\dots, J-1$, we obtain
\begin{equation}\label{eq:con22-3}
  \|e^{j+1}_{n+1}\|_X \le  \frac{\tilde\delta}{2\sqrt\alpha} +  \frac{\alpha}{\alpha+ \sigma^2}\|e^j_{n+1}\|_X  + \left(\frac{9d_1d_3}{4\alpha} + \frac{d_3}{4\sqrt\alpha} \right)\|e^j_{n+1}\|_X^2. 
\end{equation}
For $\tilde \delta\le \tilde\delta_0$, we have from \eqref{eq:alpha} and \eqref{eq:delta0} that
\begin{equation*}
 \left(\frac{\tilde \delta}{2\xi \left[1 - \frac{\epsilon(2+d_0)}{3}\right] d_0 c_0}\right)^{2/3} \le \alpha \le \frac{\epsilon}{3-\epsilon}\sigma^2.
\end{equation*}
Or, equivalently, $\alpha$ satisfies \eqref{eq:con18} and the following inequality
 \begin{equation}\label{eq:con34}
 \frac{\tilde \delta}{2\sqrt\alpha}\le \xi \left[1 - \frac{\epsilon(2+d_0)}{3}\right]d_0c_0\alpha.
 \end{equation}
On the other hand, there exists $N_0$ such that  condition \eqref{eq:con20}  is satisfied for all $N > N_0$ and
 \begin{equation}\label{eq:con21}
    \Delta k\left( d_0 + \frac{d_2}{2\sqrt\alpha}\right) \le (1 - \xi)\left[1 - \frac{\epsilon(2+d_0)}{3}\right]d_0c_0\alpha,
 \end{equation}
Now using the same arguments as in the proof of Theorem \ref{the:con1}, we can show that $\|e^j_{n}\|_X \le d_0c_0 \alpha$ for all $j = 0,\dots, J; n = 1,\dots, N$, if \eqref{eq:cond0} is satisfied. This implies the positivity of $r_n$ as in Theorem \ref{the:con1}. Moreover,
\begin{equation}\label{eq:con501}
\begin{split}
\|e^1_{n+1}\|_X &\le \Delta k\left( d_0 + \frac{d_2}{2\sqrt\alpha}\right)  + \frac{\tilde \delta}{2\sqrt\alpha} + \frac{\epsilon(2+d_0)}{3} \|e_n\|_X,\\
\|e^{j+1}_{n+1}\|_X &\le \frac{\tilde \delta}{2\sqrt\alpha} + \frac{\epsilon(1+d_0)}{3} \|e^j_{n+1}\|_X, j = 1,\dots, J-1.
\end{split}
\end{equation}
Consequently, for $n = 0,\dots, N-1,$ we have 
\begin{equation}\label{eq:con502}
 \|e_{n+1} \|_X \le \frac{\tilde \delta}{2\sqrt\alpha}\frac{1}{1 - \frac{\epsilon(1+d_0)}{3}} + \left( \frac{\epsilon(1+d_0)}{3}\right)^{J-1} \left[ \Delta k\left( d_0 + \frac{d_2}{2\sqrt\alpha}\right) +  \frac{\epsilon(2+d_0)}{3} \|e_n\|_X \right].
\end{equation}
Hence, 
\begin{eqnarray}
 \|e_{N}\|_X  &\le& \frac{\tilde \delta}{2\sqrt\alpha \left[ 1 - \frac{\epsilon(1+d_0)}{3} \right] \left[1 -\left(\frac{\epsilon(1+d_0)}{3}\right)^{J-1}\frac{\epsilon(2+d_0)}{3}\right]} \nonumber\\
  &+&  \left(\frac{\epsilon(1+d_0)}{3}\right)^{J-1} C_1\sqrt{\Delta k}.\label{51}
\end{eqnarray}
Here the constant $C$ is the same as in Theorem \ref{the:con1}.   Finally, taking into account the condition (\ref{eq:con34}) we obtain (\ref{eq:con38}) with the constant $C_2$ given by
$$
C_2 = \frac{(\xi d_0 \frac{4\epsilon}{27d_1d_3})^{1/3}}{\left[2\left( 1 - \frac{\epsilon(1+d_0)}{3} \right)  \right]^{2/3} \left[1 -\left(\frac{\epsilon(1+d_0)}{3}\right)^{J-1}\frac{\epsilon(2+d_0)}{3}\right]}.
$$
The proof is complete.
\end{proof}
\begin{remark}\label{trade-off}

To obtain the H\"older type error estimate of the form $\|e_{N}\|_X = O(\tilde \delta^{2/3}) $, 
we require that $\left(\frac{\epsilon(1+d_0)}{3}\right)^{J-1} \sqrt{\Delta k} = O(\tilde \delta^{2/3})$. 
That means, if $\Delta k$ is small, we do not need to use many Newton iterations and vice-verse. 
In other words, there is a trade-off between the frequency step $\Delta k$ and the number of Newton iterations for a given accuracy.   

\end{remark}

\subsection{Discussion on the link between the true shape and the observable shapes}\label{discussion--}

 Theorems \ref{the:con1} and \ref{th:con2} show the accuracy of the reconstruction of the observable shapes. The final accuracy of the algorithm with respect to the true shape depends on the stability of the reconstruction problem under investigation. When the final frequency $k_h$ is very high,   a H\"older type stability estimate was proved in \cite{S-T:IPI2012} for the part of the boundary illuminated by the incident wave. Hence a natural question arises: are the error estimates of Theorems \ref{the:con1} and \ref{th:con2}  uniform with respect to frequency interval when $k_h$ becomes very large. The answer to this question depends on the dependence of the constants $C_1$ and $C$ on the frequency interval. For simplicity, we assume that the lowest frequency is fixed and the same frequency step is used in all frequency intervals. Below we give a heuristic, non-rigorous explanation about which factors could affect the error estimates when $k_h$ increases. 
 
 First of all, we know that the higher the frequency, the better the stability of the reconstruction problem. Therefore the observable shapes should become closer and closer. As a result, the constant $d_0$ in Assumption 2 should not increase when $k_h$ is increased. Second, we can see from \eqref{eq:con4} that $d_1$, $d_2$ and $d_3$ are non-decreasing. Moreover, since $c_0$ can be bounded from above by a constant which is not increased when $k_h$ increases, see \eqref{eq:c02}. Therefore, the constant $C_2$ is non-increasing. 
 
 Concerning the constant $C_1$, from \eqref{eq:con24} it follows that the second term is non-increasing. Indeed, for a given frequency step $\Delta k$, we have
 \begin{equation*}
  \frac{N\left(\frac{\epsilon(2+d_0)}{3}\right)^N}{k_h - k_l} =  \frac{\left(\frac{\epsilon(2+d_0)}{3}\right)^N}{\Delta k}.
 \end{equation*}
That means, it is non-increasing when $k_h$ increases if the frequency step is kept fixed. The other factors of the second term of \eqref{eq:con24} are clearly non-increasing. Hence, the only factor which could cause the constant $C_1$ to increase is $d_2$ in the first term of \eqref{eq:con24}. 

The question on how this factor $d_2$ depends on the frequency is still open to us. 
Note however that, based on integral equation methods, precisely the explicit dependence of the norms of the corresponding boundary integral operators
 in terms of the frequencies, see \cite{Ch-Mo:2008, M-Me} for instance, we infer that $d_2$ 
increases as $k_h$ increases, but at a moderate rate, i.e. polynomially. 
Then we can eliminate its effect on the constant $C_1$ by increasing the number of Newton steps at each frequency. 
 We will investigate this question in a future work.

\section{Multi-level Newton method}\label{sec:mNm}

In this section, we discuss how to obtain the comparable error estimates as in the previous section but with a less restrictive condition than \eqref{eq:cond0} concerning the reconstruction at the lowest frequency. 
For this purpose, we use a multi-level Newton method which is described hereafter. 

We recall that the error estimate \eqref{eq:con38} was obtained under the conditions \eqref{eq:con18}, \eqref{eq:con20}, \eqref{eq:con34} and \eqref{eq:con21}. 
In this section, we choose $\xi=1/2$ for simplicity. To make the analysis easier to follow, we rewrite the above conditions here
\begin{eqnarray}
 &&\alpha \le  \frac{\epsilon\sigma^2}{3 - \epsilon},\label{eq31} \\
&&\frac{9d_1d_4}{4\alpha}\Delta k \le \frac{\epsilon}{3},\label{eq32}\\
 &&   \Delta k\left( d_0 + \frac{d_2}{2\sqrt\alpha}\right) \le \frac{1}{2}\left[1 - \frac{\epsilon(2+d_0)}{3}\right]d_0c_0\alpha,\label{eq33}\\
 &&\frac{\tilde \delta}{2\sqrt\alpha}\le  \frac{1}{2} \left[1 - \frac{\epsilon(2+d_0)}{3}\right]d_0c_0\alpha, \label{eq34}
\end{eqnarray}
with the constant $c_0$ being given by \eqref{eq:c0} which depends on $\alpha$. Therefore, in the following, we denote by $c_0(\alpha)$ to indicate this dependence. We reserve the notations $c_0$ and $\alpha$ for the constants in the previous section, i.e.~these constants associate with the full frequency set. So Theorem \ref{th:con2} says that if the conditions \eqref{eq31}--\eqref{eq34} are satisfied, and if the solution $r_0$ at the lowest wavenumber $k_0$ satisfies \eqref{eq:cond0}, i.e. 
\begin{equation}\label{eq362}
    \|\tilde r(k_{l}) - r_0\|_X \le d_0c_0\alpha < \tilde c,
\end{equation}
then the final error estimate \eqref{eq:con38} holds true.

We remark that the regularization parameter $\alpha$ depends on the smallest singular value $\sigma$ of the domain derivative $\tilde A_{n}, n = 0, 1, \dots, N$. Clearly, the more frequencies used, the smaller this singular value $\sigma$ is. Therefore, by subdividing the original interval of frequencies into sub-intervals and choosing this regularization parameter depending on the smallest singular value in different frequency sub-intervals, we may not need to choose a small regularization parameter (in other words, with a less restrictive condition on the initial guess) at the first sub-interval but still obtain a comparable error estimate as \eqref{eq:con38}. 

To make the following analysis consistent with the previous section, we still consider the set of frequencies $k_0 = k_l, \dots, k_N = k_h$ with step size $\Delta k$ as in section \ref{sec:RLA}. Suppose that the original frequency interval $\{k_0,\dots, k_N\}$ is divided into $M$ sub-intervals from low to high frequencies. These sub-intervals do not need to have the same number of frequencies. We denote by $\tilde \sigma_m$ the smallest singular value in the $m$-th sub-interval. That is,
$$
\tilde\sigma_m = \min\{\sigma_n, k_n \text{ belongs to the $m$-th sub-interval}\}.
$$
Here $\sigma_n$ the smallest singular value of $\tilde A_n\big\vert_{X_{n+1}}$ as in section \ref{sec:RLA}. 
Moreover, we choose the sequence of parameters $\hat \sigma_m, m = 1, \dots, M$, as follows:
$$
\hat\sigma_1 = \tilde \sigma_1, \quad \hat\sigma_{m+1} = \min\{\hat \sigma_m, \tilde \sigma_{m+1}\}, m = 1,\dots, M-1.
$$
by this choice of the parameters $\hat \sigma_m$, it is clear that 
\begin{equation}\label{eq:sigmam}
   \hat \sigma_1 \ge \hat\sigma_2\ge \cdots \ge \hat\sigma_M \ge \sigma.
\end{equation}
Associated with these sub-intervals, we choose the set of regularization parameters $\alpha_m$, $m = 1,\dots, M$ such that \eqref{eq31} is satisfied in each sub-intervals, 
where $\sigma$ is replaced by the corresponding parameter $\hat \sigma_m$. That is,
\begin{equation}
   \alpha_m \le  \frac{\epsilon\hat\sigma_m^2}{3 - \epsilon}, m = 1, \dots, M.\label{eq312}
\end{equation}
Moreover, $\alpha_m$ are also chosen such that 
\begin{equation}\label{eq:alpham}
   \alpha_1 \ge \alpha_2\ge \cdots \ge \alpha_M \ge \alpha.
\end{equation}
The multi-level Newton algorithm can be summarized as follows.

\begin{algorithm}\label{alg:2}
$ $
\begin{itemize}
\item Given measured data  $u^{\infty,\delta}_m(\cdot,k)$ for $k =
k_0, \dots, k_N$, and the partition of this frequency interval into $M$ sub-intervals.
\item Step 1: find an approximation $r_0 \in X_0$ at frequency $k_0$.
\item Step 2: For $m = 1,\dots, M$
    \begin{itemize}
    \item Choose $\alpha_m$ satisfying \eqref{eq312} and \eqref{eq:alpham}.
    \item Use Algorithm \ref{alg:1} to find an approximation in the $m$-th frequency sub-interval. 
    \end{itemize}
\end{itemize}
\end{algorithm}

Let us show a similar convergence result as in Theorem \ref{th:con2} for this algorithm. From \eqref{eq:c0} and \eqref{eq:alpham} it can be proved using elementary analysis that 
\begin{equation}\label{eq352}
 c_0(\alpha_1) \alpha_1 \ge c_0(\alpha_2) \alpha_2 \ge \cdots \ge c_0(\alpha_M) \alpha_M \ge c_0 \alpha.
\end{equation}
We recall that $c_0$ and $\alpha$ are associated with the whole frequency interval $\{k_0,\dots, k_N\}$. It also follows from \eqref{eq:alpham} and \eqref{eq352} that the inequalities \eqref{eq32}--\eqref{eq34} still hold for the same frequency step $\Delta k$ and noise level as in Theorems \ref{the:con1} and \ref{th:con2} when $\alpha$ is replaced by $\alpha_m $ and $c_0$ by $c_0(\alpha_m)$. That means, all the conditions of these theorems are satisfied for each sub-interval.

Now we replace the condition \eqref{eq:cond0} in Theorems \ref{the:con1} and \ref{th:con2} by the following one for the first sub-interval:
\begin{equation}\label{eq:a1}
    \|\tilde r(k_{l}) - r_0\|_X \le d_0c_0(\sigma_1)\alpha_1 < \tilde c.
\end{equation}
 Hence, from Theorem \ref{th:con2} (see \eqref{51}) we obtain the following error estimate in the first sub-interval
\begin{eqnarray}
 \|\tilde r(k_{N_1}) - r_{N_1}\|_X  &\le& \frac{\tilde \delta}{2\sqrt\alpha_1 \left[ 1 - \frac{\epsilon(1+d_0)}{3} \right] \left[1 -\left(\frac{\epsilon(1+d_0)}{3}\right)^{J-1}\frac{\epsilon(2+d_0)}{3}\right]} \nonumber\\
  &+&  \tilde C_1\left(\frac{\epsilon(1+d_0)}{3}\right)^{J-1} \sqrt{\Delta k}.\nonumber \\
  &\le& \frac{\tilde \delta}{2\sqrt\alpha \left[ 1 - \frac{\epsilon(1+d_0)}{3} \right] \left[1 -\left(\frac{\epsilon(1+d_0)}{3}\right)^{J-1}\frac{\epsilon(2+d_0)}{3}\right]} \nonumber\\
  &+&  \tilde C_1\left(\frac{\epsilon(1+d_0)}{3}\right)^{J-1} \sqrt{\Delta k}.\label{eq36}
\end{eqnarray}
 for a constant $\tilde C_1$.  This constant can be chosen fixed for all frequency sub-intervals and independent of $\tilde\delta$. Here $k_{N_1}$ is the maximum frequency of the first sub-interval. It follows from (\ref{eq33}) and \eqref{eq36} that
\begin{equation}
 \|\tilde r(k_{N_1}) - r_{N_1}\|_X  \le\frac{d_0c_0\alpha}{2\left[1 -\left(\frac{\epsilon(1+d_0)}{3}\right)^{J-1}\frac{\epsilon(2+d_0)}{3}\right]}+  \tilde C_1\left(\frac{\epsilon(1+d_0)}{3}\right)^{J-1} \sqrt{\Delta k}.\label{eq:a2}
\end{equation}
In the second sub-interval, we use the final approximation $r_{N_1}$ of the first sub-interval as the initial guess, i.e., it plays the same role as $r_0$ in section \ref{sec:RLA}. For the given frequency step $\Delta k$, we can choose the number of Newton iterations $J$ large enough so that the following inequality holds true
\begin{equation}\label{eq:a3}
  \|\tilde r(k_{N_1}) - r_{N_1}\|_X \le d_0c_0\alpha_2.
\end{equation}
This process can be continued until the last sub-interval. In the last sub-interval, we obtain a similar error estimate as \eqref{eq:con38}. We summarize the above analysis in the following theorem.
\begin{theorem}\label{th:multi}
 Suppose that the frequency set $\{k_0, \dots, k_N\}$ is subdivided into $M$ sub-intervals. Denote by $N_m$ the number of frequencies in the $m$-th sub-interval.  Moreover, let $\epsilon$ be a positive real number satisfying $0 < \epsilon < 3/(2+d_0)$, and $\alpha_m$, $m = 1,\dots, M$, be the regularization parameters satisfying \eqref{eq312} and \eqref{eq:alpham}. We also suppose that the frequency step is small enough so that the conditions \eqref{eq32} and \eqref{eq33} are fulfilled for $\alpha = \min\{\alpha_m, \ m = 1, \dots, M\}$ and $c_0 = c_0(\alpha)$ given by \eqref{eq:c0}.  
Then there exists an integer $J$ large enough such that if the reconstruction at the lowest frequency satisfies \eqref{eq:a1},  we have $r_{n}(t) > 0\  \forall  t\in [0,2\pi]$ and the following error estimate holds true
 \begin{equation}\label{eq:con3822}
    \|\tilde r(k_{h}) - r_N\|_X \le C_2 \tilde \delta^{2/3} +  \left(\frac{\epsilon(1+d_0)}{3}\right)^{J-1} \tilde C_1 \sqrt{\Delta k},
  \end{equation}
  for every $\tilde\delta \le \tilde\delta_0$, where $C_2$ is as in Theorem \ref{th:con2} and $\tilde C_1$ is a constant independent of $\tilde\delta$, $\alpha_m$ and $N$. Here $\tilde\delta_0$ is defined as in \eqref{eq:delta0}.

 \end{theorem}

\begin{remark}
 Theorem \ref{th:multi} indicates that we still obtain the same error estimate as in Theorem \ref{th:con2} with $C_1$ being replaced by $\tilde C_1$. 
 That means, by using the multi-level algorithm, we can obtain basically the same error estimate as in Theorem \ref{th:con2} with the first guess $r_0$
satisfying the condition \eqref{eq:a1} which is, in general, weaker than \eqref{eq:cond0} due to \eqref{eq352}. 
This issue is related to estimating the lower bounds of the singular values $\sigma_1$. Actually, at each level $m$, $m=1, ..., M$, we take the regularization parameter
$\alpha$ satisfying similar estimate, i.e (\ref{eq312}). In a forthcoming work, we will investigate the lower bound of $\sigma_m$ in terms of the frequency $\kappa$ and the dimension of the 
corresponding space $X_{n+1}$. With such estimates at hand, the regularization parameter $\alpha_m$ can be chosen compared to the known quantities $k$ and $n$. Let us finally make some comments on the condition
 (\ref{eq34}) on the noise level. As the frequency becomes high, $\alpha$ becomes small and so for the noise level. However this is quite natural since at high frequencies we expect to reconstruct
 small details and this makes sense only if the measurements at hand are not so noisy.
\end{remark}

\section{Numerical results}\label{sec:num}
 In this section, we show some numerical results to demonstrate the performance of the proposed algorithms. We also compare reconstruction results using these algorithms with the recursive linearization algorithm. 

In these tests, we considered flower-shaped obstacles defined by the equation 
$$
 \{ x(t) =
c_1(1 + c_2 \cos c_3 t)(\cos t, \sin t) , t\in [0,2\pi) \}
$$
with
positive constants $c_1$, $c_2$ and $c_3$. The first parameter
determines the area of the obstacle, the second one relates to the
curvature and the last one determines the number of petals of the
"flower". Two obstacles were considered which correspond to two sets of parameters: $c_1 = 2$, $c_2 = 0.3$ and $c_3 = 4$ (obstacle 1), and $c_1 = 2$, $c_2 = 0.2$ and $c_3 = 9$ (obstacle 2).

The measured far field patterns $u_m^\infty(\cdot,k_j,r)$, $j =
0,\dots, N$, used in these tests were simulated as the solution of
the forward problem \eqref{eq:1}--\eqref{eq:3} which was solved by
the integral equation method~\cite{C-K:1998}. We used $16$ observation
directions uniformly distributed on the unit circle. The same method was
also used to calculate the domain derivative of the far field
operator. Additive random
noise of  $5\%$ was added to the computed far field patterns.

Our numerical tests in \cite{S-T:IPI2012} have indicated that although the regularization parameter $\alpha$ must satisfies conditions \eqref{eq:con18} and \eqref{eq:con34} in the theoretical analysis, numerical performance seemed to be more optimistic. In our tests, this parameter could be chosen in a wide range, say, from $10^{-6} $ to
$10^{-1}$ which still provided good reconstruction results. Therefore, in the following examples, the regularization parameter $\alpha$ was chosen to be $10^{-2}$. We fixed the direction of incidence to be $\theta = (-\frac{1}{2},\frac{\sqrt{3}}{2}) $.  The wavenumbers were chosen between $k_l = 0.5$ and $k_h = 8$. 

The approximation $r_0$ at the lowest frequency was computed as follows: we first approximated the obstacle by a circle. In this case, the center and radius of the approximating circle were found by minimizing the corresponding least-squares objective functional. Then, three Fourier coefficients were chosen to represent $r_0$. These parameters were then found by solving again the least-squares minimization problem using the Matlab optimization routine \textit{fmincon}. Our results showed that these optimization problems were stable, therefore we did not need a good initial guess. However, we obtained only an approximation of the low frequency information of the obstacle in this step.  

\begin{figure}[ht]
 \centering 
 \subfloat[]{\includegraphics[width= 0.48\textwidth,height= 0.44\textwidth]{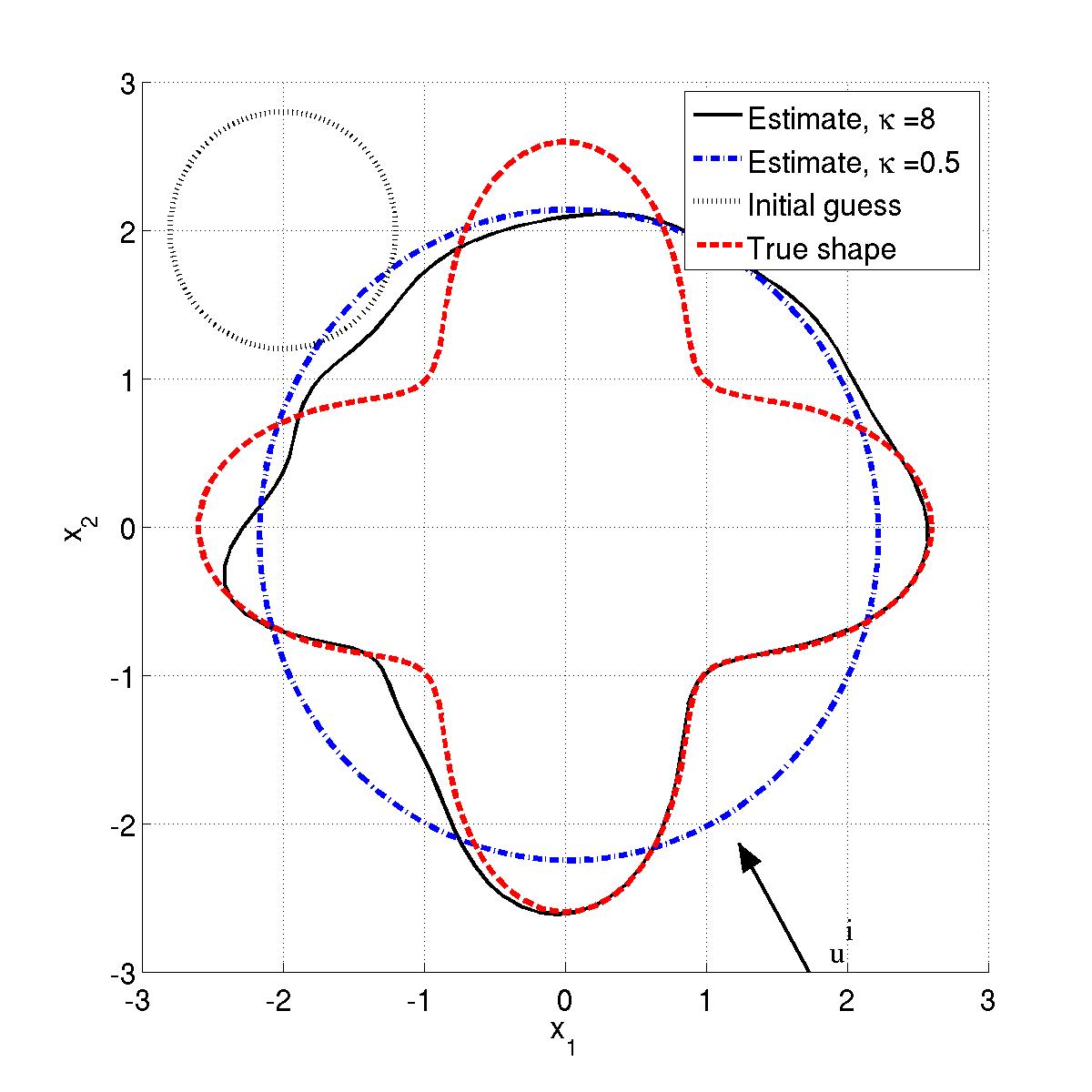}}
 \subfloat[]{\includegraphics[width= 0.48\textwidth, height=0.44\textwidth]{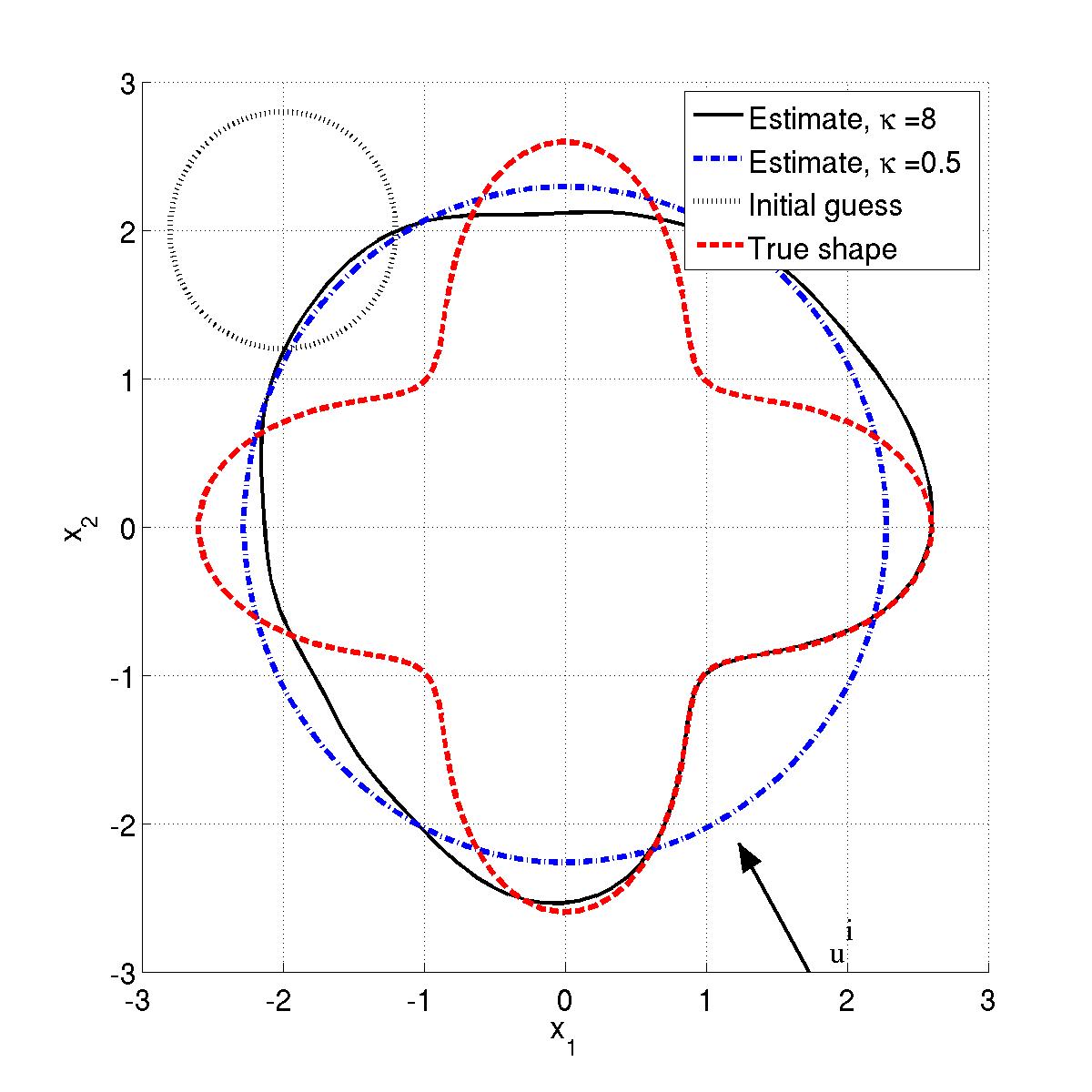}}
 \caption{Reconstruction of obstacle 1 using 12 wavenumbers: (a)  4 Newton iterations; (b) 1 Newton iteration.} 
 \label{fig:1}
\end{figure}

Figure \ref{fig:1} shows the reconstruction of obstacle 1 using 12 wavenumbers. In Figure \ref{fig:1}(a), 4 Newton iterations at each frequency were used while only 1 Newton iteration at each frequency was used in Figure \ref{fig:1}(b). We can see that the first one is more accurate than the second one. We remark that, as pointed out in \cite{S-T:IPI2012}, the reconstruction is good in the part of the obstacle illuminated by the incident plane wave but the details of the shadow part is not well reconstructed.   

\begin{figure}[ht]
 \centering 
 \subfloat[]{\includegraphics[width= 0.48\textwidth,height= 0.44\textwidth]{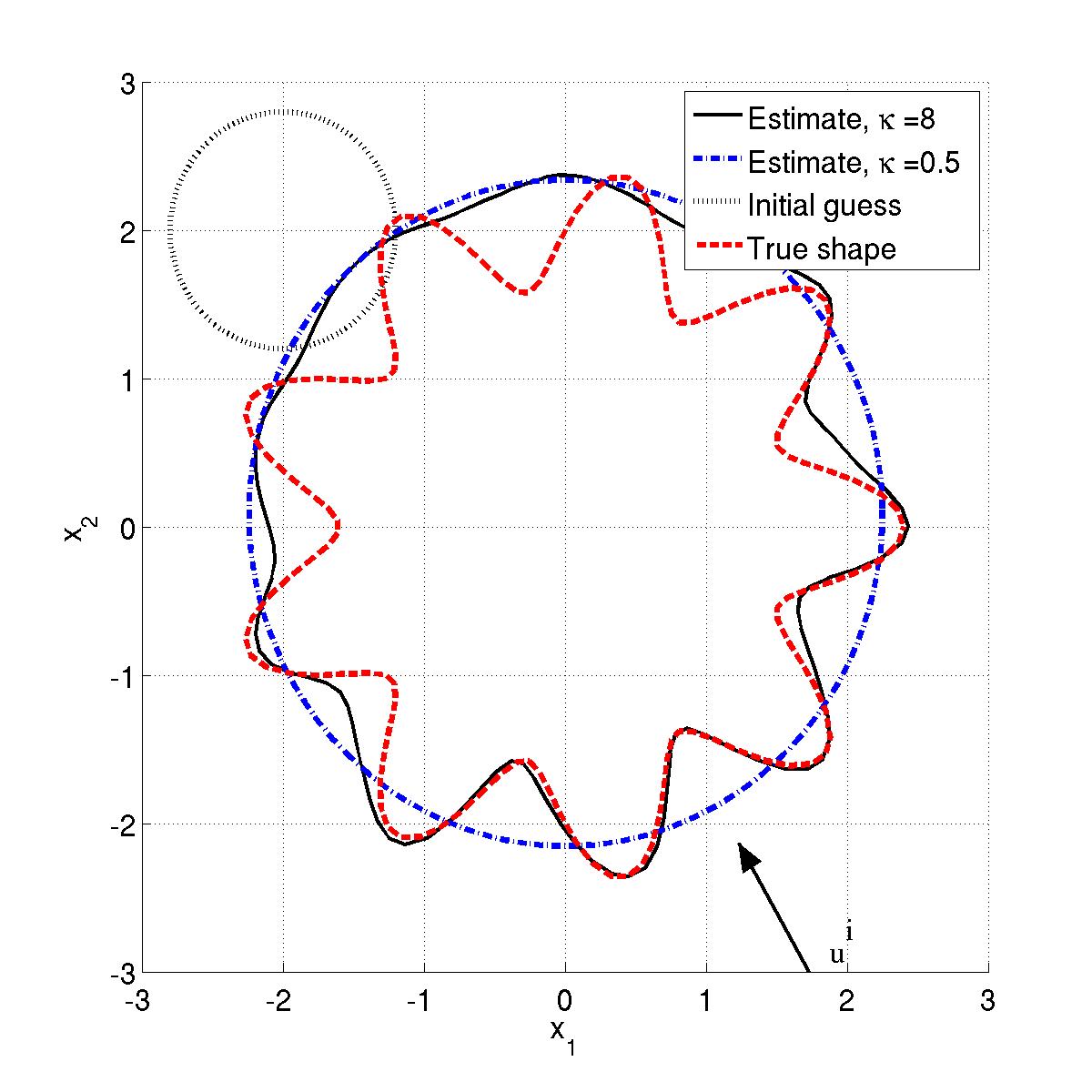}}
 \subfloat[]{\includegraphics[width= 0.48\textwidth, height=0.44\textwidth]{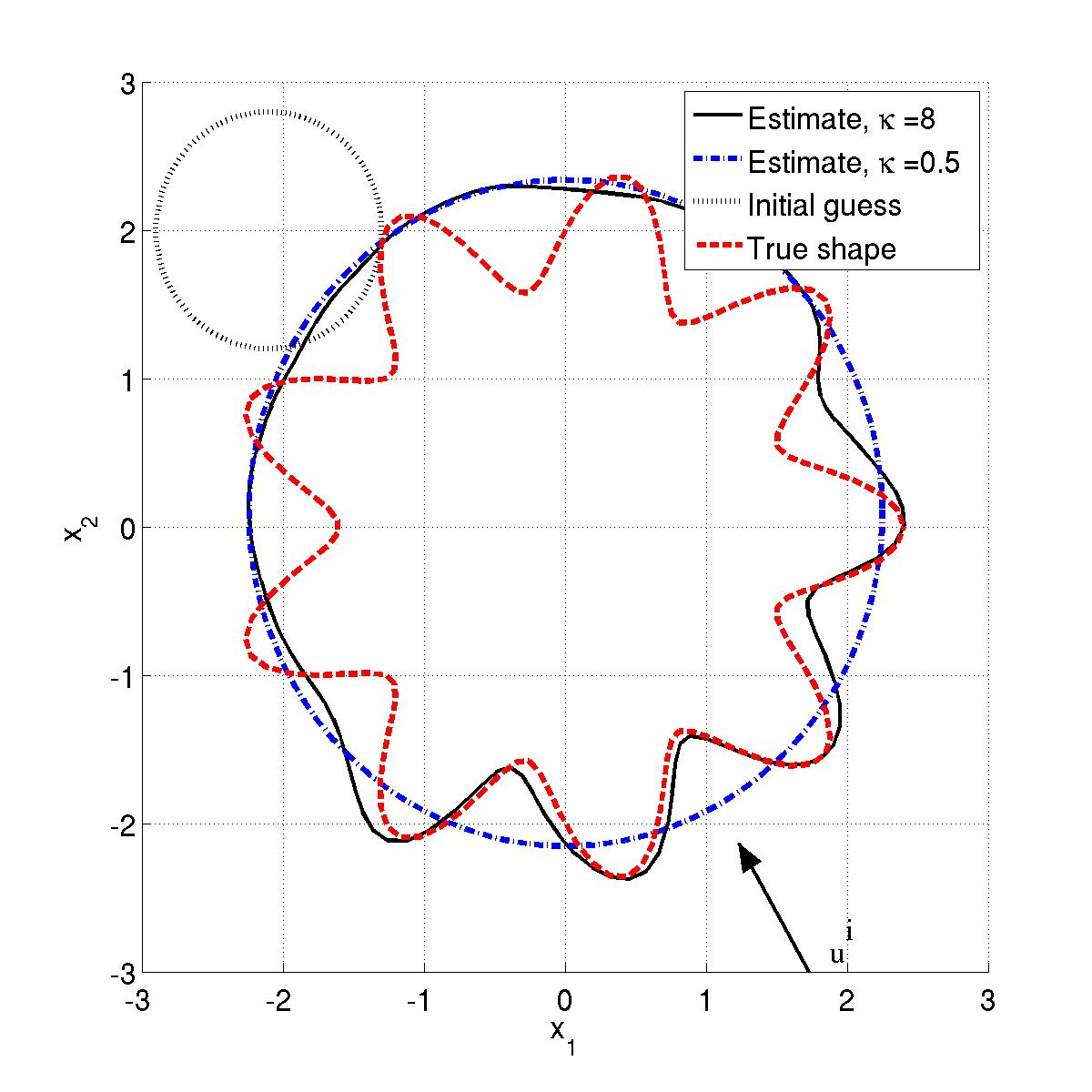}}
 \caption{Reconstruction of obstacle 2 using 20 wavenumbers: (a) 4 Newton iterations; (b) 1 Newton iteration.} 
 \label{fig:2}
\end{figure}

In Figure \ref{fig:2} we depict the reconstruction of obstacle 2. For this obstacle, 20 wavenumbers were used in order to reconstruct its small detailed features. We also can see that using 4 Newton iterations improved the accuracy compared to using only 1 iteration. We would like to emphasize that this improvement is more clear for a smaller number of frequencies or a larger number of Newton iterations, see Figure \ref{fig:4} for the results of obstacle 2 using 16 wavenumbers.
  
\begin{figure}[ht]
 \centering 
  \subfloat[]{\includegraphics[width= 0.48\textwidth,height= 0.44\textwidth]{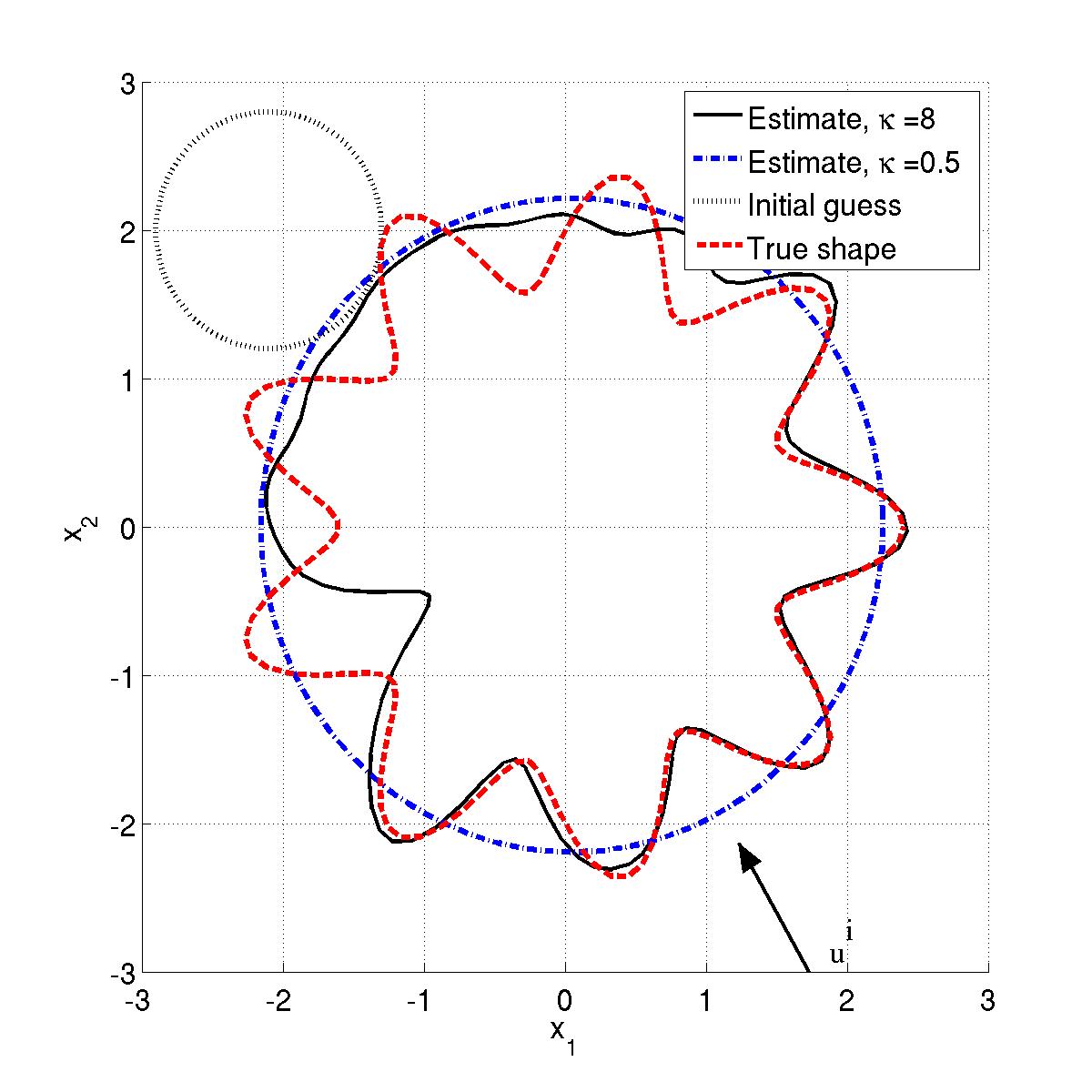}}
  \subfloat[]{\includegraphics[width= 0.48\textwidth, height=0.44\textwidth]{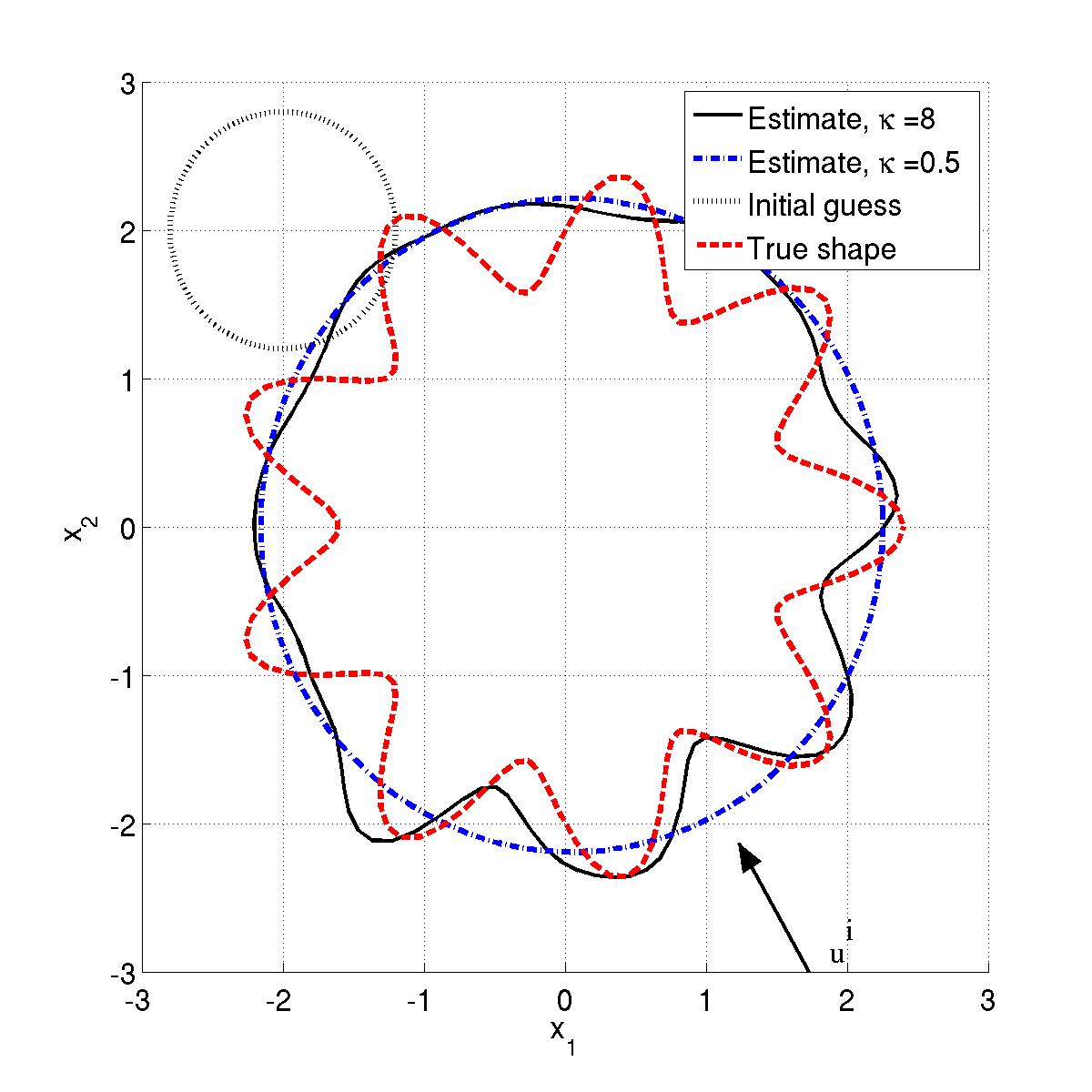}}
  \caption{Reconstruction of obstacle 2 using 16 wavenumbers: (a) 10 Newton iterations; (b) 1 Newton iteration.} 
 \label{fig:4}
\end{figure}

\begin{figure}[ht]
 \centering 
  \subfloat[]{\includegraphics[width= 0.48\textwidth,height= 0.44\textwidth]{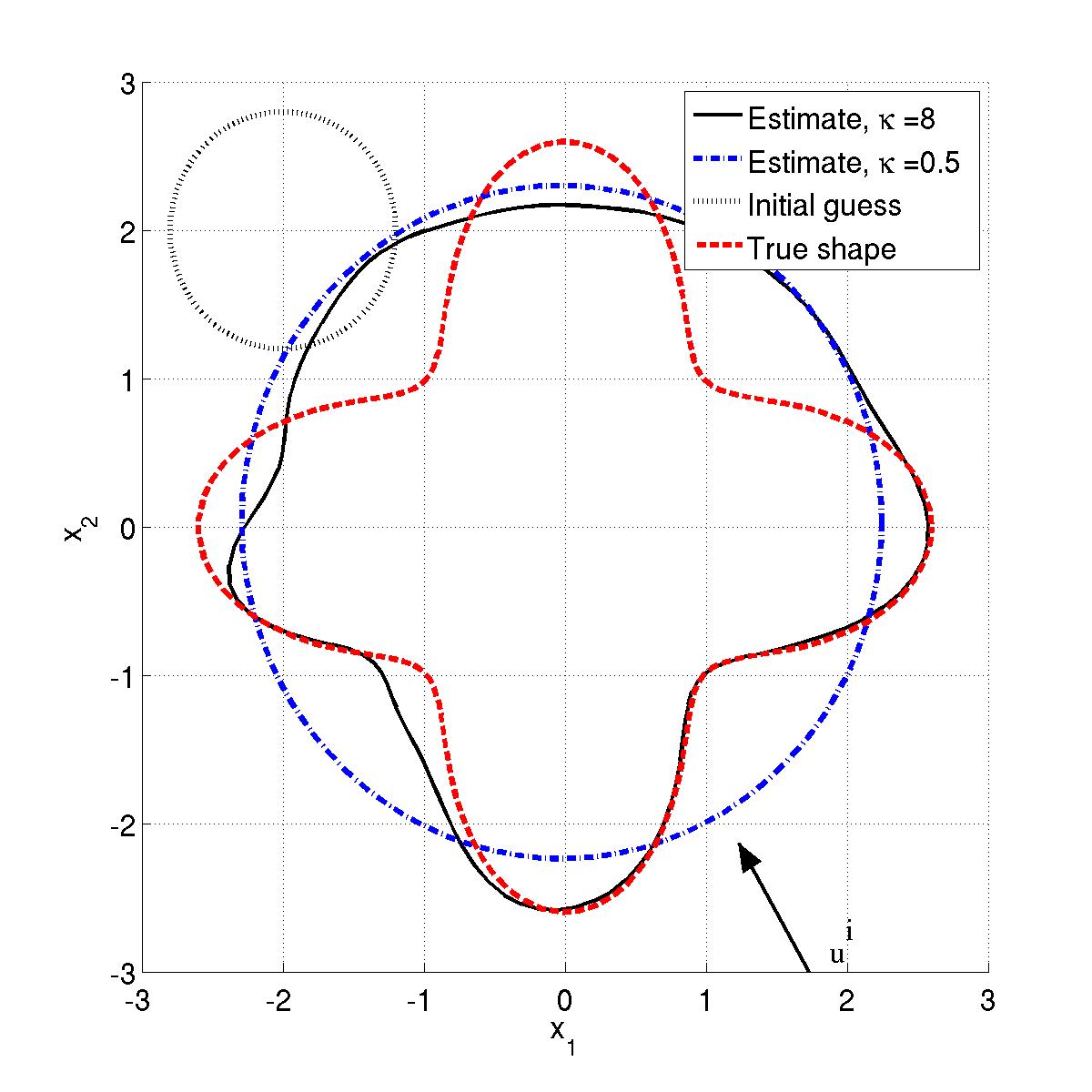}}
  \subfloat[]{\includegraphics[width= 0.48\textwidth, height=0.44\textwidth]{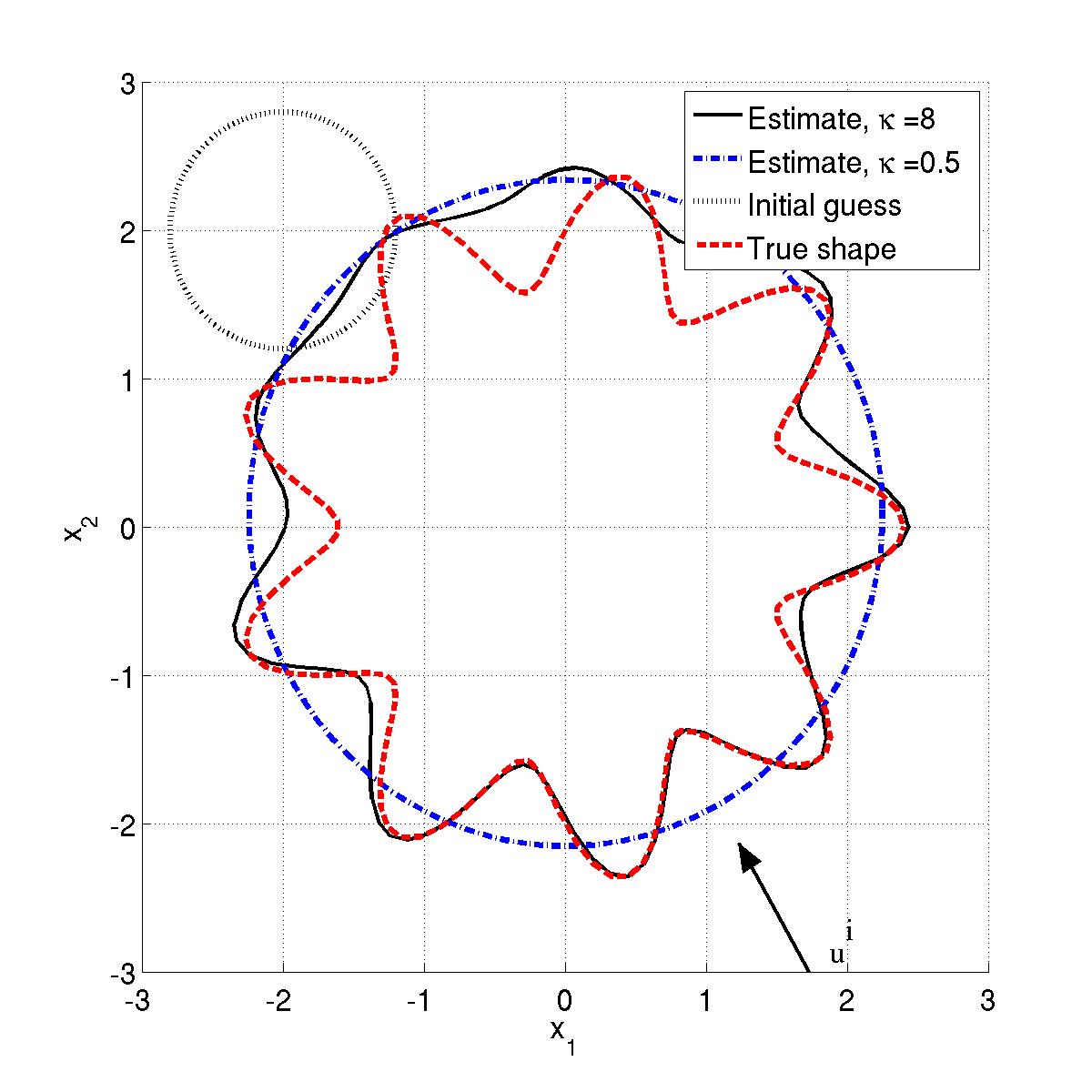}}
  \caption{Reconstruction using the multi-level Newton method: (a) obstacle 1 with 12 wavenumbers; (b) obstacle 2 with 20 wavenumbers.} 
 \label{fig:3}
\end{figure}

To see the performance of the multi-level method of Section \ref{sec:mNm}, we show in Figure \ref{fig:3} the reconstruction of the two obstacles. The reconstruction $r_0$ at the lowest frequency was obtained by just one iteration of the nonlinear least-squares optimization problem. By doing so, we expected that this should not be as good as in the previous tests. The regularization parameter $\alpha$ at the first frequency step was chosen to be $0.04$ which is 4 times larger than that at the other frequencies. Moreover, 5 iterations were used at the first step and 4 iterations were used at the other frequencies. As can be seen, the reconstructions are comparable to Figure \ref{fig:1}(a) and Figure \ref{fig:2}(a) which confirm our theoretical analysis.

\section*{Acknowledgments}
M.~Sini was supported by the Johann Radon Institute for
Computational and Applied Mathematics (RICAM), Austrian Academy of
Sciences and by the Austrian Science Fund (FWF) under the project No.~P22341-N18.
N.~T.~Th\`anh was supported by US Army Research Laboratory and US Army
Research Office grants W911NF-11-1-0399.
%

\end{document}